\documentclass[psamsfonts]{amsart}

\usepackage{amssymb,amsfonts}
\usepackage{amsmath}
\usepackage[all,arc]{xy}
\usepackage{enumerate}
\usepackage{mathrsfs}
\usepackage{mathtools}
\usepackage{cite}
\usepackage[unicode]{hyperref}
\usepackage{bookmark}
\newtheorem{thm}{Theorem}[section]
\newtheorem{theorem}[thm]{Theorem}

\newtheorem{prop}[thm]{Proposition}
\newtheorem{lemma}[thm]{Lemma}

\newtheorem{claim}[thm]{Claim}
\newtheorem{fact}[thm]{Fact}

\theoremstyle{definition}
\newtheorem{defn}[thm]{Definition}

\newcommand{\defeq}{\mathrel{\mathop:}=}

\newcommand{\R}{\mathbb{R}}

\newcommand{\Q}{\mathbb{Q}}
\newcommand{\PP}{\mathbb{P}}

\newcommand{\LL}{\mathbb{L}}
\newcommand{\CC}{\mathbb{C}}

\newcommand{\FF}{\mathbb{F}}

\newcommand{\la}{\lambda}
\newcommand{\ka}{\kappa}
\newcommand{\w}{\omega}
\newcommand{\sm}{\setminus}
\newcommand{\mc}{\mathcal}

\DeclareMathOperator{\cf}{cf}
\DeclareMathOperator{\cof}{cof}

\DeclareMathOperator{\ITP}{ITP}
\DeclareMathOperator{\TP}{TP}

\DeclareMathOperator{\ot}{ot}
\DeclareMathOperator{\Coll}{Col}

\DeclareTextCommand{\textaleph}{L8U}{ℵ}

\newcommand{\uhr}{\upharpoonright}
\newcommand{\forces}{\Vdash}

\makeatletter
\let\c@equation\c@thm
\makeatother
\numberwithin{equation}{section}

\title{Strong Tree Properties Along Many Segments of Successors of Singulars}
\author{William Adkisson}
\begin{document}

\begin{abstract}
	The strong tree property and the super tree property (also called ITP) are generalizations of the tree property that characterize strong compactness and supercompactness up to inaccessibility. That is, an inaccessible cardinal $\kappa$ is strongly compact if and only if the strong tree property holds at $\ka$, and supercompact if and only if ITP holds at $\kappa$.
	Generalizing a result of Golshani and Hayut, we show that from large cardinals it is consistent for ITP to hold simultaneously at any countable initial segment of successors of singular cardinals. More formally, given any countable ordinal $\theta$, we construct a forcing extension in which ITP holds at the first $\theta$ successors of singulars. We then extend this result further to obtain the strong tree property on long segments of successors of singular cardinals of multiple cofinalities simultaneously.
\end{abstract}

\maketitle

\section{Introduction}
The tree property, an uncountable generalization of K\"onig's Lemma, holds at a cardinal $\ka$ if every tree of height $\ka$ with levels of size $<\ka$ has a cofinal branch. The tree property at uncountable cardinals is closely connected with large cardinals: if $\ka$ is inaccessible, then the tree property holds at $\ka$ if and only if $\ka$ is weakly compact. Moreover, if the tree property holds at a small cardinal like $\aleph_2$, that cardinal must be weakly compact in $L$.

A long-running project in set theory, begun by Magidor in the 1970s, is obtaining the tree property at all regular cardinals above $\aleph_1$ simultaneously. Since the tree property always fails at $\aleph_1$ and at singular cardinals, this question asks if it is consistent for the tree property to hold at all possible cardinals simultaneously. This can be seen as a way of measuring how much large cardinal strength it is possible to have in a model of ZFC: if all possible cardinals have the tree property, than each carries some remnant of the strength of a weakly compact cardinal.

Since this question was posed, progress has been slow but steady. There have been too many results to give a full account of the history, but we will list some of the highlights. From countable many supercompacts, Cummings and Foreman \cite{CF_TreeProp} forced the tree property at each $\aleph_n$ for $1<n<\w$ simultaneously; Neeman \cite{NeemanTPNw+1} extended this result to include the tree property at $\aleph_{\w+1}$. The current record is due to Cummings, Hayut, Magidor, Neeman, Sinapova, and Unger \cite{cummings:tpNw2+4}, who obtained the tree property up to $\aleph_{\w^2+4}$, with $\aleph_{\w^2}$ strong limit.

Specialized techniques are required to obtain the tree property at successors of singular cardinals, which tend to behave in a different manner than successors of regular cardinals. Magidor and Shelah \cite{MS_TPSuccSing} showed that it was consistent for the tree property to hold at $\aleph_{\w+1}$ from very strong large cardinal hypotheses; these large cardinal assumptions were improved by Sinapova \cite{sinapova:tpatNw+1} to countably many supercompacts. As part of obtaining the tree property up to $\aleph_{\w+1}$ in \cite{NeemanTPNw+1}, Neeman refined these techniques, giving a general class of forcings that will obtain the tree property at $\aleph_{\w+1}$.

Obtaining the tree property at many successors of singular cardinals simultaneously requires more complicated constructions. This was first achieved by Golshani and Hayut \cite{golshani-hayut:tpcountablesegment}, who (assuming the existence of many supercompact cardinals) obtained the tree property at an arbitrary countable initial segment of successors of singular cardinals. That is, given any countable ordinal $\theta$, they built a forcing extension in which the tree property holds at the first $\theta$ successors of singular cardinals.

Magidor's question can also posed for several generalizations of the tree property. Of particular interest are the strong tree property and the super tree property (also called ITP). These properties characterize strong compactness and supercompactness up to inaccessibility in the same way that the tree property characterizes weak compactness. Recall that if $\ka$ is inaccessible, then $\ka$ is weakly compact if and only if the tree property holds at $\ka$. Jech \cite{jech:combinatorialprobs} showed that if $\ka$ is inaccessible, then the strong tree property holds at $\ka$ if and only if $\ka$ is strongly compact; Magidor \cite{magidor:combcharsc} proved that if $\ka$ is inaccessible, then it is supercompact if and only if the super tree property (also known as ITP) holds at $\ka$. In the absence of canonical inner models for strongly compacts or supercompacts, the presence of the associated tree properties is some of the best heuristic evidence we have for large-cardinal lower bounds of properties like the proper forcing axiom.

While these properties were originally formulated in the 70s, there has been a renewal of interest in the past few decades. In particular, Magidor's problem can also be posed for these strengthenings of the tree property, and many constructions forcing the tree property at many cardinals simultaneously have been extended to the strong tree property and ITP. Fontanella\cite{fontanella:CF} and Unger\cite{UngerCF} independently generalized the Cummings and Foreman result to ITP, showing that it is consistent for ITP to hold at $\aleph_n$ for $1<n<\w$. Focusing on successors of singular cardinals, Sinapova and Hachtman \cite{HachtmanITPNw+1} showed that from countably many supercompacts, it is consistent for ITP to hold at $\aleph_{\w+1}$. Building on their techniques, the author \cite{adkisson:ITP} described a general class of forcings that can obtain the strong and super tree properties at successors of a singular cardinals, and used this to show that in Neeman's construction, ITP holds at each $\aleph_n$ and the strong tree property holds at $\aleph_{\w+1}$.

In this paper, we generalize Golshani and Hayut's result to these stronger properties, showing that (assuming the existence of  many supercompacts) it is consistent that ITP can hold at arbitrary countable initial segments of successors of singular cardinals. More precisely, given $\theta < \w_1$, we will construct a model in which $\ITP(\aleph_{\alpha+1})$ holds for all limit ordinals $\alpha < \theta$.

Golshani and Hayut's result focuses on many successors of singulars of the same cofinality; it is also natural to examine the tree property and its strengthenings at successors of singulars of many cofinalities simultaneously.
In \cite{adkisson:manycofs}, the author forced the strong tree property to hold at successors of singulars of many cardinals simultaneously. The construction in that paper was directly inspired by the arguments of Golshani and Hayut, and uses very similar techniques, but applies them to a different part of the forcing. In Section \ref{s:manycofs}, we combine the two constructions, obtaining the strong tree property at long segments of singulars of many cofinalities simultaneously.

\section{Preliminaries}

We begin by defining the strong and super tree properties. These properties are similar to the standard tree property, which states that every tree property of height $\ka$ with levels of size $<\ka$ has a cofinal branch. The primary difference is that the strong and super tree properties are concerned with a more general class of tree-like objects, called thin $\mc{P}_\ka(\la)$-lists.

\begin{defn}
	Let $\ka$ be a regular cardinal and let $\la \geq \ka$. A sequence $d = \langle d_z \mid z\in \mc{P}_\ka(\la)\rangle$ is a \emph{$\mc{P}_\ka(\la)$-list} if for all $z \in \mc{P}_\ka(\la)$, $d_z\subseteq z$. 
	We define the $z$-th level of a list, denoted $L_z$, by
	\[L_z = \{d_y \cap z \mid z\subseteq y, y \in \mc{P}_\ka(\la)\}.\]
	A $\mc{P}_\ka(\la)$-list is \emph{thin} if $|L_z| < \ka$ for all $z \in \mc{P}_\ka(\la)$.
\end{defn}

%Note that if $\ka$ is inaccessible, every $\mc{P}_\ka(\la)$ list is thin.

\begin{defn}
	A set $b\subseteq \la$ is a \emph{cofinal branch} through a $\mc{P}_\ka(\la)$-list $d$ if $b\cap z \in L_z$ for all $z\in \mc{P}_\ka(\la)$.
	
	We say that $\TP(\ka,\la)$ holds if every thin $\mc{P}_\ka(\la)$-list has a cofinal branch.
	We say that the \emph{strong tree property} holds at $\ka$ if $\TP(\ka,\la)$ holds for all $\la \geq \ka$.
\end{defn}

To verify the strong tree property, it suffices to examine unboundedly many $\la$:

\begin{fact}\cite{fontanella:stpsuccsing}
	Let $\la' > \la$. Then $\TP(\ka, \la')$ implies $\TP(\ka,\la)$.
\end{fact}

The super tree property (also called ITP), is similar to the strong tree property, except that we have more stringent requirements on the desired branch.

\begin{defn}
	A set $b \subseteq \la$ is an \emph{ineffable branch} through a $\mc{P}_\ka(\la)$-list $d$ if $\{z \in \mc{P}_\ka(\la) \mid b \cap z = d_z\}$ is stationary.
	
	We say $\ITP(\ka,\la)$ holds if every thin $\mc{P}_\ka(\la)$-list has an ineffable branch. We say that \emph{ITP holds at $\ka$}, or that $\ITP(\ka)$ holds, if $\ITP(\ka,\la)$ holds for all $\la \geq \ka$.
\end{defn}

Note that every ineffable branch is cofinal, so $\ITP(\ka)$ implies the strong tree property at $\ka$. Like with the strong tree property, $\ITP(\ka,\la)$ increases in strength as the second coordinate increases:

\begin{fact}\cite[Proposition 3.4]{weiss:combinatorialessence}\label{fact:ITPgoesdownwards}
	Let $\la' > \la$. $\ITP(\mu, \la')$ implies $\ITP(\mu, \la)$.
\end{fact}

In the specific case where $\ka = \la$, it is often convenient to restrict to elements of $\ka$ rather than of $\mc{P}_\ka(\ka)$. The definitions are analogous:

\begin{defn}
	Let $\ka$ be a regular cardinal. A sequence $d = \langle d_\alpha \mid \alpha < \ka\rangle$ is a \emph{$\ka$-list} if $d_\alpha \subseteq \alpha$ for all $\alpha < \ka$. The \emph{$\alpha$-th level}, denoted $L_\alpha$, is the set $\{d_\beta \cap \alpha \mid \alpha < \beta < \ka\}$. A $\ka$-list is \emph{thin} if every level has size $<\ka$. A set $b \subseteq \ka$ is a \emph{cofinal branch} through a list $d$ if $b\cap \alpha \in L_\alpha$ for all $\alpha < \ka$, and is an \emph{ineffable branch} if $\{\alpha < \ka \mid b\cap \alpha = d_\alpha\}$ is stationary in $\ka$.
\end{defn}

Since $\ka$ is club in $\mc{P}_\ka(\ka)$, we have the following easy fact:

\begin{fact}\label{fact:ka-lists(ka,ka)-lists}
	$\TP(\ka,\ka)$ is equivalent to the statement that every thin $\ka$-list has a cofinal branch. Similarly, $\ITP(\ka, \ka)$ is equivalent to the statement that every thin $\ka$-list has an ineffable branch.
\end{fact}

The main tool to work with cofinal and ineffable branches is the (thin) $\ka$-approximation property.

\begin{defn}\label{def:approxed}
	Let $\ka$ be regular, $\la$ be an ordinal, and $\PP$ be a forcing notion in a model $V$. A $\PP$-name $\dot{b}$ for a subset of $\la$ is \emph{$\ka$-approximated by $\PP$ over $V$} if for all $z \in (\mc{P}_{\ka}(\la))^V$, $\forces_\PP \dot{b}\cap z \in V$.
	
	A $\PP$-name $\dot{b}$ for a subset of $\la$ is \emph{thinly $\ka$-approximated by $\PP$ over $V$} if it is $\ka$-approximated by $\PP$ over $V$, and furthermore for every $z \in (\mc{P}_{\ka}(\la))^V$, $|\{x \in V \mid \exists p\in \PP \ p \forces_\PP x = \dot{b}\cap z\}|<\ka$.
\end{defn}

\begin{defn}\label{def:ka-approx}
	Let $\ka$ be regular. A forcing $\PP$ has the \emph{$\ka$-approximation property} over a model $V$ if for every ordinal $\la$ and $\PP$-name $\dot{b}$ for a subset of $\la$, if $\dot{b}$ is $\ka$-approximated by $\PP$ over $V$, then $\forces_{\PP}\dot{b} \in V$.
	
	A forcing $\PP$ has the \emph{thin $\ka$-approximation property} over $V$ if for every ordinal $\la$ and every $\PP$-name $\dot{b}$ for a subset of $\la$, if $\dot{b}$ is thinly $\ka$-approximated by $\PP$ over $V$, then $\forces_{\PP}\dot{b} \in V$.
\end{defn}
Note that the $\ka$-approximation property implies the thin $\ka$-approximation property. 

These properties, which which appear in a number of contexts in set theory, are useful here because any cofinal or ineffable branch through a thin $\mc{P}_\ka(\la)$ list is always thinly approximated over $V$, by any forcing.

\begin{fact}\cite{adkisson:ITP}[Lemma 2.8]\label{fact:branchapprox}
	Let $d$ be a thin $\mc{P}_\ka(\la)$ list in $V$, and let $\mathbb{P}$ be a notion of forcing over $V$. Suppose $\dot{b}$ is a $\PP$-name for a cofinal branch through this list. Then $\dot{b}$ is thinly $\ka$-approximated by $\PP$ over $V$.
\end{fact}

If a forcing is split into multiple stages, however, branches are not guaranteed to be approximated by forcings in the intermediate stages. To ensure that they are approximated at every step, we use the following lemma. This was proven for the thin approximation property by Unger \cite[Lemma 7.10]{UngerCF}.

\begin{lemma}\label{lem:cctoapprox}\cite[Lemma 2.9]{adkisson:ITP}
	Suppose that $\dot{b}$ is a $\PP\ast\dot{\Q}$-name for a subset of some ordinal $\la$, which is (thinly) $\ka$-approximated by $\PP\ast \dot{\Q}$ over $V$. If $\PP$ has the $\ka$-cc, then in $V[\PP]$, $\dot{b}$ is still (thinly) $\ka$-approximated by $\Q$.
\end{lemma}

The easiest way to verify approximation properties is the following lemma, due to Unger.

\begin{lemma}\cite[Lemma 2.4]{UngerAtrees}\label{lem:squarecctoapproxprop}
	Let $\ka$ be regular. Let $\PP$ be a forcing such that $\PP\times \PP$ is $\ka$-cc. Then $\PP$ has the $\ka$-approximation property.
\end{lemma}

At the successor of a singular cardinal $\nu$, we use the following lemma to show that forcings have the thin $
\nu^+$ approximation property. This is a generalization of the classical branch lemma \cite[Lemma 2.2]{MS_TPSuccSing} of Magidor and Shelah.

\begin{lemma}\cite[Lemma 2.20]{adkisson:ITP} \label{lem:M-S}
	Suppose $\nu$ is a singular strong limit cardinal with cofinality $\tau$. Let $\Q$ be a $\mu^+$-closed forcing over a model $V$ for some $\mu < \nu$ with $\tau \leq \mu$, and let $\PP \in V$ be a poset with $|\PP| <\mu^+$. Then $\Q$ has the thin $\nu^+$-approximation property in the generic extension of $V$ by $\PP$.
\end{lemma}

To obtain the tree property, Golshani and Hayut's result makes central use of \cite[Theorem 3.10]{NeemanTPNw+1}. This theorem which gives a general class of posets that will force the tree property at the successor of a singular cardinal of countable cofinality.
To obtain the strong tree property, we will use the following variation of that theorem. It is a corollary of \cite[Theorem 3.10]{NeemanTPNw+1}, since the tree property at a cardinal $\kappa$ is equivalent to the statement that every thin $\ka$-list has a cofinal branch; it also follows (even more immediately) from \cite[Theorem 3.14]{adkisson:ITP}, which gives the full strong tree property at $\nu^+$ under these hypotheses.

\begin{lemma}\label{lem:gen2cardstrTP}
	Let $\tau$ be a regular cardinal. Let $\langle \ka_\rho \mid \rho < \tau \rangle$ be an increasing continuous sequence of cardinals above $\tau$ with supremum $\nu$, such that $\ka_{\rho}^+ = \ka_{\rho+1}$ for all $\rho < \tau$. Let $I$ be a subset of $\ka_0$ such that every $\mu \in I$ has cofinality $\tau$. For each $\mu \in I$, let $\LL_\mu$ be the product of forcings $\PP_\mu$ and $\Q_\mu$, where $|\PP_\mu| < \mu^+$ and $\Q_\mu$ is $\mu^{+}$-closed, such that $|\LL_\mu| < \ka_{\rho'}$ for some fixed $\rho' < \tau$.
	In addition, suppose that we have the following:
	\begin{itemize}
		\item $\ka_0$ is indestructibly $\nu^+$-supercompact, with a normal measure $U_0$ on $\mc{P}_{\ka_0}(\nu^+)$ and corresponding embedding $i$ such that $\nu \in i(I)$.
		\item For all ordinals $\rho < \tau$ and all $\la \geq \nu^+$, there is a generic $\la$-supercompactness embedding $j_{\rho+2}$ with domain $V$ and critical point $\ka_{\rho+2}$, added by a poset $\FF$ such that the full support power $\FF^{\ka_\rho}$ is $<\ka_{\rho}$-distributive in $V$.
	\end{itemize}
	Then there exists $\mu \in I$ such that $\TP(\nu^+,\nu^+)$ holds in the extension of $V$ by $\LL_\mu$.
\end{lemma}

For the super tree property, we will use a similar lemma. The hypotheses are almost identical: the only change is that we require each $\Q_\mu$ to be $\mu^{++}$-closed rather than $\mu^+$-closed. (This assumption, while seemingly mild, is what prevents the arguments in Section \ref{s:manycofs} from immediately generalizing to the super tree property.) We also note that while this lemma gives full ITP, we will only need $\ITP(\nu^+, \nu^+)$ for this paper.

\begin{lemma}\cite[Theorem 4.6]{adkisson:ITP}\label{lem:gen2cardITP}
	Let $\tau$ be a regular cardinal. Let $\langle \ka_\rho \mid \rho < \tau \rangle$ be an increasing continuous sequence of cardinals above $\tau$ with supremum $\nu$, such that $\ka_{\rho}^+ = \ka_{\rho+1}$ for all $\rho < \tau$. Let $I$ be a subset of $\ka_0$ such that every $\mu \in I$ has cofinality $\tau$. For each $\mu \in I$, let $\LL_\mu$ be the product of forcings $\PP_\mu$ and $\Q_\mu$, where $|\PP_\mu| < \mu^+$ and $\Q_\mu$ is $\mu^{++}$-closed, such that $|\LL_\mu| < \ka_{\rho'}$ for some fixed $\rho' < \tau$.
	In addition, suppose that we have the following:
	\begin{itemize}
		\item $\ka_0$ is indestructibly $\nu^+$-supercompact, with a normal measure $U_0$ on $\mc{P}_{\ka_0}(\nu^+)$ and corresponding embedding $i$ such that $\nu \in i(I)$.
		\item For all ordinals $\rho < \tau$ and all $\la \geq \nu^+$, there is a generic $\la$-supercompactness embedding $j_{\rho+2}$ with domain $V$ and critical point $\ka_{\rho+2}$, added by a poset $\FF$ such that the full support power $\FF^{\ka_\rho}$ is $<\ka_{\rho}$-distributive in $V$.
	\end{itemize}
	Then there exists $\mu \in I$ such that $\ITP$ holds at $\nu^+$ in the extension of $V$ by $\LL_\mu$.
\end{lemma}

With these general lemmas established, we turn our attention to the specific forcing posets that we will use. As in \cite{golshani-hayut:tpcountablesegment}, our construction primarily consist of a product of Easton collapses. The Easton collapse was introduce by Shioya in \cite{Shioya:eastoncolls}, and has better projection properties than the standard Levy collapse.

\begin{defn}
	A cardinal $\gamma$ is \emph{strong regular} if $\gamma^{<\gamma} = \gamma$.
\end{defn}
\begin{defn}
	The Easton collapse $E(\mu,\nu)$ is the product 
	\[\prod_{\mu\leq \gamma < \nu,\ \gamma \text{ strong regular}} \Coll(\mu,\gamma)\]
	with Easton support.
\end{defn}

We record some properties of this poset.

\begin{lemma}\cite{Shioya:eastoncolls}\label{lem:eastonprops}
	Let $\mu$ be a regular cardinal, and let $\nu$ be Mahlo. Then $E(\mu,\nu)$ is $\mu$-closed, $\nu$-Knaster, and $|E(\mu,\nu)| = \nu$.
\end{lemma}

We note that even if $\nu$ is not Mahlo, $E(\mu,\nu)$ remains $\mu$-closed.

\begin{lemma}\cite[Lemma 2.5]{golshani-hayut:tpcountablesegment}\label{lem:projections}
	Let $\langle \mu_i \mid i<\zeta\rangle$ and $\langle \nu_i \mid i<\zeta\rangle$ be increasing sequences of regular cardinals, and assume that $\zeta < \mu_0$ and $\mu_i \leq \nu_i < \mu_{i+1}$. Let $\lambda \geq \sup \mu_i$. Then there is a projection from $E(\mu_0, \la)$ onto the full support product $\Pi_{i<\zeta} E(\mu_i,\nu_i)$.
\end{lemma}

\begin{lemma}\cite[Lemma 2.6]{golshani-hayut:tpcountablesegment}\label{lem:liftingembeddings}
	Let $\ka$ be indestructibly supercompact and let $\mu<\ka$ regular. Let $\PP_0,$ and $\PP_1$ be forcing notions, where $|\PP_0| < \mu$, $\PP_0$ is $\mu$-cc, and $\PP_1$ is $\ka$-directed closed. Then in the generic extension by $\PP_0 \times E(\mu,\ka)\times \PP_1$, $\ka$ is generically supercompact by a poset $\R \in V$ that is $\mu$-closed in the generic extension.
\end{lemma}

Combining Lemma \ref{lem:liftingembeddings} with Lemma \ref{lem:gen2cardITP}, we conclude:

\begin{prop}\label{prop:itpaftereastoncolls_withaux}
	Let $\langle \ka_n \mid n<\w\rangle$ be an $\w$-sequence of indestructibly supercompact cardinals with supremum $\nu$.
	For all strong regular $\la > \sup(S)$, define
	\[\PP_\la =  \prod_{n<\w}E(\ka_n^{++},\ka_{n+1}) \times E(\nu^{++}, \sup(S))\times \Coll(\nu^+,\la),\]
	and for each $\rho < \ka_0$, let
	\[\LL_\rho = \Coll(\w,\rho) \times \Coll(\rho^{++}, <\ka_0).\]
	Then for all strong regular $\la > \sup(S)$, there is $\rho < \ka_0$ such that
	%for all strong regular $\la \geq \sup(S)$,
	$\ITP(\nu^+)$ holds in the generic extension of $V[\PP_\la]$ by $\LL_\rho$.
\end{prop}

This proposition is actually stronger than we will need; we will only require the one-cardinal version, $\ITP(\nu^+, \nu^+)$, to hold in this extension.

In addition to Lemma \ref{lem:projections}, we will also use the following standard absorption lemma for Levy collapse posets.
\begin{lemma}\label{lem:colabsorption}\cite[Theorem 14.3]{cummings:handbook}
	Let $\ka$ be an inaccessible cardinal, and let $\delta < \ka$ be regular. Let $\PP$ be a $\delta$-closed forcing poset with $|\PP| < \ka$. Then there is a forcing projection from $\Coll(\delta, <\ka)$ to $\PP$ whose quotient is $\Coll(\delta, <\ka)$.
\end{lemma}

Finally, we will need a lemma to show that certain forcings do not add small sets. The traditional lemma of this form is Easton's lemma:

\begin{lemma}[Easton]
	Let $\PP$ be a $\ka$-cc forcing and $\Q$ be a $\ka$-closed forcing. Then $\PP$ forces that $\Q$ is $<\ka$-distributive.
\end{lemma}

Unfortunately, Easton's Lemma will not apply directly; instead, we need to prove a variant in which $\PP$ is replaced with a full support product of forcings with good chain condition and closure.

\begin{lemma}\label{lem:easton}
	Let $\nu$ be a strong limit singular cardinal, and let $\langle \ka_\alpha \mid \alpha<\tau\rangle$ be an increasing sequence of regular cardinals cofinal in $\nu$. 
	Let $\CC$ be a $\nu^+$-closed forcing, and let $\PP_\tau  = \prod_{\alpha < \tau} \PP_\alpha$ be the full support product of forcings $\PP_\alpha$.
	
	Suppose that for unboundedly many $\alpha<\tau$, $\PP_{<\alpha}$ is $\ka_\alpha^+$-cc and $\PP_{\geq \alpha}$ is $\ka_\alpha^+$-closed.
	Then $\PP_\tau$ forces that $\CC$ is $<\nu^+$-distributive.
\end{lemma}
\begin{proof}
	Let $\dot{f}$ be a $\PP_\tau \times \CC$-name for a function from $\nu$ to the ordinals. Let $I\subseteq \tau$ be the unbounded set given by the hypotheses of the lemma. For each $\alpha \in I$, let $\dot{f}_\alpha$ be a $\PP_\tau\times \CC$-name forced by the empty condition to be the restriction of $\dot{f}$ to domain $\ka_\alpha$.
	
	By assumption, for all $\alpha \in I$, $\PP_{<\alpha}$ is $\ka_\alpha^+$-cc and $\PP_{\geq \alpha}$ is $\ka_\alpha^+$-closed. Applying the standard version of Easton's Lemma, we see that each $\PP_{\geq \alpha}$ is $<\ka_\alpha^+$-distributive.
	Since each $\dot{f}_\alpha$ is a name for a sequence of size $\ka_\alpha$, $\PP_{\geq \alpha}$ could not have added $\dot{f}_\alpha$, so we can find a $\PP_{<\alpha}$-name $\dot{f}'_\alpha$ for $\dot{f}\uhr \ka_\alpha$.
	
	Doing this for all $\alpha$ in $I$, we can define a $\PP_\tau$-name $\dot{f}' = \bigcup_{\alpha < \nu} \dot{f}'_\alpha$. Note that we are mildly abusing notation here: each condition $p$ appearing in each name $\dot{f}_\alpha$ is an element of $\PP_{<\alpha}$, not $\PP_\tau$, but we can replace it with the condition $p'\in \PP_\tau$
	that agrees with $p$ on all coordinates where $p$ is defined, and is trivial on the rest.
	
	Clearly $1_{\PP_\tau\times \CC} \forces \dot{f} = \dot{f}'$. We conclude that $\CC$ cannot have added $f$, since we have found a $\PP_\tau$-name for $f$.
\end{proof}
	
\section{ITP on a countable segment of successors of singular cardinals}

In this section we strengthen Golshani and Hayut's result to the super tree property. A key idea in their construction is to, for each limit ordinal $\alpha$, build a forcing $\PP_\alpha$ obtaining the tree property at $\alpha^+$. This forcing is designed to project onto an initial segment of the final forcing $\PP_t$. Given an $\alpha^+$-tree in $\PP_t$, they show that the tail of $\PP_t$ is closed enough that the tree must have been added by this initial segment. Since $\alpha^+$-trees are upwards absolute, they apply the tree property in $\PP_{\alpha}$ to obtain a branch, and then argue that this branch is present in $\PP_t$.

There are two primary difficulties in generalizing their argument. The first is that unlike trees, thin $\mc{P}_\ka(\la)$-lists are not upwards absolute, so ITP in an outer model will not a priori ensure the existence of ineffable branches through lists in the ground model. The second issue is that thin $\mc{P}_{\alpha^+}(\la)$-lists are very large, so we won't have enough closure to pass lists down to an initial segment of the final forcing.

To solve these problems, we add two extra pieces to the forcing. To deal with the lack of absoluteness, we introduce an auxiliary forcing that collapses $\la$ to $\ka$. This allows us restrict our attention to thin $\ka$-lists, which are upwards absolute. To solve the second issue, we modify $\PP_\alpha$ to make it project onto all of $\PP_t$. This is done by adding a single large Easton collapse that will project onto a tail of the target forcing.

\begin{thm}
	Let $S$ be a set of indestructibly supercompact cardinals with $\ot(S) = \min(S)^+$. For each countable ordinal $\theta$, there exists a generic extension in which the super tree property holds at all cardinals of the form $\aleph_{\alpha+1}$ for every limit ordinal $\alpha < \theta$.
\end{thm}
\begin{proof}
	Let $\ka_0 = \min(S)$. Let $D = \lim(S)\cap \sup S \cap \cof(\w)$. For every $\alpha \in D$, let $s_\alpha$ be an $\w$-sequence cofinal in $\alpha$, with $\ka_0 = \min(s_\alpha)$.
	For each $\rho < \ka_0$, define 
	\[\LL_\rho = \Coll(\w,\rho)\times \Coll(\rho^{++}, <\ka_0).\]
	For every $\alpha \in D$ and every strongly regular $\la > \sup(S)$, define
	\[\PP_{\alpha,\la} = \left(\prod_{n<\w} E(s_\alpha(n)^{++}, s_\alpha(n+1))\right) \times E(\alpha^{++}, \sup(S))\times \Coll(\alpha^+, \la).\]
	For all such $\alpha$ and $\la$, we apply Proposition \ref{prop:itpaftereastoncolls_withaux} to conclude that there is $\rho_{\alpha,\la} < \ka_0$ such that after forcing with
	$\LL_{\rho_{\alpha,\la}} \times \PP_{\alpha,\la}$, ITP holds at $\alpha^+$.
	
	For each fixed $\alpha \in D$, there is some $\rho_\alpha < \ka_0$ such that for unboundedly many $\lambda > \sup(S)$, $\rho_{\alpha,\la} = \rho_\alpha$.
	By Fodor's Lemma, there is a fixed $\rho < \ka_0$ and $S' \subseteq D$ stationary such that for all $\alpha \in S'$, $\rho_{\alpha} = \rho$.
	
	Note that for all $\alpha \in S'$, $\rho = \rho_{\alpha,\lambda}$ for unboundedly many $\lambda$, but the values of $\lambda$ for which this equality holds depend on $\alpha$. So it is not necessarily the case that for fixed $\alpha,\beta \in S'$, $\rho_{\alpha,\lambda} = \rho_{\beta,\lambda}$ for unboundedly many (or even any) $\la$.
	
	Fix $\theta < \w_1$. As in \cite{golshani-hayut:tpcountablesegment}, we find a subsequence of $S\cup S'$ of order type $\theta$, using the following fact:
	\begin{fact}\cite[Lemma 3.1]{golshani-hayut:tpcountablesegment}\label{fact:t}
		Suppose $\theta < \w_1$. Then there exists a closed $t \subseteq S$ with $\ot(t) = \theta$ such that for all $\alpha \in \lim(t)$, $\alpha \in S'$ and $s_\alpha \subseteq t$.
	\end{fact}
	Let $t$ be as in Fact \ref{fact:t}, and define $\PP_t$ to be the forcing poset
	\[\PP_t = \prod_{i<\theta} E(t(i)^{++},t(i+1)).\]
	Let $L_\rho \times G_t$ be $\LL_\rho \times \PP_t$-generic over $V$. $V[L_\rho \times G_t]$ is our final model; we will show that in this model, $\ITP(\aleph_{\alpha+1})$ holds for every limit ordinal $\alpha < \theta$. Note that for each such $\alpha$, $(\aleph_\alpha)^{V[L_\rho \times G_t]} = t(\alpha)$ and $(\aleph_{\alpha+1})^{V[L_\rho \times G_t]} = t(\alpha)^+$.
	
	Let $\alpha < \theta$ be a limit ordinal. Let $\la > \sup(S)$ be strongly regular such that $\rho = \rho_{\alpha,\la}$. By Fact \ref{fact:ITPgoesdownwards}, since this holds for unboundedly many $\la$, it suffices to verify $\ITP(t(\alpha)^+, \la)$ only for these values of $\la$. Let $G_{t(\alpha),\la}$ be generic for $\PP_{t(\alpha),\la}$, and let $\dot{d}$ be a $\LL_\rho\times \PP_t$-name for a thin $\mc{P}_{t(\alpha)^+}(\la)$-list.
	
	Let $K$ be generic for $\Coll(t(\alpha)^+, \la)$ over $V$. Work in $V[L_\rho][G_t][K]$. Note that we can split $\PP_t$ into the following product:
	\[\PP_t = \left(\prod_{i < \alpha} E(t(i)^{++}, t(i+1))\right) \times \left(\prod_{\alpha \leq i} E(t(i)^{++}, t(i+1))\right).\]
	
	Let $\PP_{t,1}$ denote the first component of this product, with generic $G_{t,1}$; let $\PP_{t,2}$ denote the second component, with generic $G_{t,2}$.
	
	Next, we verify that $\Coll(t(\alpha)^+,\la)^V$ remains distributive after forcing with $\PP_t$. This is where we need to apply the generalization of Easton's Lemma. 
	
	By Lemma \ref{lem:eastonprops}, noting that we are taking a full support product, $\PP_{t,2}$ is $t(\alpha)^{++}$-closed. Again by Lemma \ref{lem:eastonprops}, each component $E(t(i)^{++}, t(i+1))$ of the first product is $t(i+1)^+$-cc and $t(i)^{++}$-closed; this remains true in $V[G_{t,2}]$. Similarly,  $\Coll(t(\alpha)^+,\la)^V$ remains closed in $V[G_{t,2}]$. Working in $V[G_{t,2}]$, we see that $\PP_{t,1}\times \Coll(t(\alpha)^+,\la)^V$ meets the hypotheses of Proposition \ref{lem:easton}. We conclude that 
	$\Coll(t(\alpha)^+, \la^+)$ remains $t(\alpha)^+$-distributive in $V[G_t]$. It follows that $\mc{P}_{t(\alpha)^+}(\la)^{V[G_t]} = \mc{P}_{t(\alpha)^+}(\la)^{V[G_t][K]}$. Since these two sets are the same, we will omit the superscripts.
	
	In $V[L_\rho][G_t][K]$, $\la$ has been collapsed to an ordinal with cardinality and cofinality $t(\alpha)^+$. In particular, this means that there is an order-isomorphism between $\mc{P}_{t(\alpha)^+}(\la)$ and $\mc{P}_{t(\alpha)^+}(t(\alpha)^+)$. Using this isomorphism, we can construct a thin $\mc{P}_{t(\alpha)^+}(t(\alpha)^+)$-list that has an ineffable branch if and only if $d$ does. In fact, since $t(\alpha)^+$ is club in $\mc{P}_{t(\alpha)^+}(t(\alpha)^+)$ and the ineffability of a branch is determined on a stationary subset, we can find a thin $t(\alpha)^+$-list $d'$ that has an ineffable branch if and only if $d$ does.
	
	We move now to $V[L_\rho][G_{t(\alpha),\la}]$. The first step is describing a projection from $V[G_{t(\alpha),\la}]$ to $V[G_t][K]$.
	
	 Recall that $\PP_{t(\alpha),\la}$ is the following product:
	\[\PP_{t(\alpha),\la} = \left(\prod_{n<\w} E(s_{t(\alpha)}(n)^{++}, s_{t(\alpha)}(n+1))\right) \times E(t(\alpha)^{++}, \sup(S))\times \Coll(t(\alpha)^+,\la).\]
	Let $\PP_{t(\alpha),\la}^1$ denote the first component, $\prod_{n<\w}E(s_{t(\alpha)}(n)^{+}, s_{t(\alpha)}(n+1))$; let $\PP_{t(\alpha),\la}^2$ denote the second component, $E(t(\alpha)^{++}, \sup(S))$.
	As before, we split $\PP_t$ into the product
	\[\PP_t = \left(\prod_{i < \alpha} E(t(i)^{++}, t(i+1))\right) \times \left(\prod_{\alpha \leq i} E(t(i)^{++}, t(i+1))\right).\]
	Let $\PP_{t}^1$ denote $\prod_{i < \alpha} E(t(i)^{++}, t(i+1))$, and $\PP_{t}^2$ denote  $\prod_{\alpha \leq i} E(t(i)^{++}, t(i+1))$. Let $G_{t(\alpha)}^1$, $G_{t(\alpha)}^2$, $G_{t}^1$, and $G_{t}^1$ be generic for $\PP_{t(\alpha),\la}^1$, $\PP_{t(\alpha),\la}^2$, $\PP_{t}^1$, and $\PP_{t}^2$, respectively.
	
	We now examine projections from each piece of $\PP_{t(\alpha),\la}$ onto the corresponding component of $\PP_t\times \Coll(t(\alpha)^+,\la)$. Clearly $\Coll(t(\alpha)^+,\la)$ projects onto itself via the identity. 
	Note that by Lemma \ref{lem:projections}, $\PP_{t(\alpha),\la}^2$ projects onto $\PP_{t}^2$.
	
	The remaining piece requires slightly more analysis. Following the arguments in \cite{golshani-hayut:tpcountablesegment}, let $\langle \xi_n \mid n < \w\rangle$ be an increasing sequence of ordinals below $\theta$ such that $t(\xi_n) = s_{t(\alpha)}(n)$. Note that we can rewrite $\PP_{t}^1$ as
	\[\prod_{n<\w} \prod_{\xi_n \leq i \leq \xi_{n+1}} E(t(i)^{++}, t(i+1)).\]
	By Lemma \ref{lem:projections}, for each $n < \w$ there is a projection
	\[E(s_{t(\alpha)}(n)^{++},s_{t(\alpha)}(n+1)) \to \prod_{\xi_n \leq i < \xi_{n+1}} E(t(i)^{++}, t(i+1));\]
	combining these gives a projection 
	\[\prod_{n<\w} E(s_{t(\alpha)}(n)^{++},s_{t(\alpha)}(n+1)) \to \prod_{n<\w} \prod_{\xi_n \leq i < \xi_{n+1}} E(t(i)^{++}, t(i+1)).\]
	That is, we obtain a projection from $\PP_{t(\alpha),\la}^1$ to $\PP_{t}^1$.
	
	Putting the pieces together, we conclude that $\PP_{t(\alpha),\la}$ projects onto $\PP_t\times \Coll(t(\alpha)^+,\la)$. Let $\PP^* = (\PP_{t(\alpha),\la}^1\times \PP_{t(\alpha),\la}^2)/\PP_t$, with generic $G^*$. Then we can write $V[L_\rho][G_{t(\alpha),\la}]$ as $V[L_\rho][G_t][K][G^*]$.

	\begin{claim}
		The list $d'$ has an ineffable branch $b'$ in $V[L_\rho][G_{t(\alpha),\la}] = V[L_\rho][G_t][K][G^*]$.
	\end{claim}
	\begin{proof}
		Since $d'$ is contained in $V[L_\rho][G_t][K]$, it must also be present in $V[L_\rho][G_{t(\alpha),\la}]$. Since being a thin $t(\alpha)^+$-list is upwards absolute, $d'$ remains a thin $t(\alpha)^+$-list in $V[L_\rho][G_{t(\alpha),\la}]$. By construction, $\ITP(t(\alpha)^+)$ holds in $V[L_\rho][G_{t(\alpha),\la}]$, so $d'$ must have an ineffable branch.
	\end{proof}

	\begin{claim}
		The branch $b'$ is in $V[L_\rho][G_t][K]$.
	\end{claim}
	Recall that the quotient $\PP^*$ is $\ka_0^{++}$-closed in $V[G_t]$, and thus also in $V[K][G_t]$. Note that $|\LL_\rho| = \ka_0^+$ (in $V[K][G_t]$), since GCH holds in $V$ and the remainder of the forcings are $\ka_0^+$. Then applying Lemma \ref{lem:M-S} in $V[G_t][K]$, we see $\PP^*$ has the thin $t(\alpha)^+$-approximation property in $V[G_t][K][L_\rho]$.
	
	By Fact \ref{fact:branchapprox}, $b'$ is thinly $t(\alpha)^+$ approximated by any forcing over $V[L_\rho][G_t][K]$. We conclude that $b' \in V[L_\rho][G_t][K]$.
	\end{proof}
	
	\begin{claim}
		The list $d$ has an ineffable branch $b$ in $V[L_\rho][G_t][K]$.
	\end{claim}
	\begin{proof}
		By the previous claims, $b'$ is contained in $V[L_\rho][G_t][K]$. Since ineffability is downwards absolute, $b'$ is an ineffable branch through $d'$. Since $d'$ has an ineffable branch if and only if $d$ does, we conclude that there is an ineffable branch $b$ through $d$ contained in $V[L_\rho][G_t][K]$.
	\end{proof}
	
	\begin{claim}
		The branch $b$ is in $V[L_\rho][G_t]$.
	\end{claim}
	\begin{proof}
		First, we note that $b$ is thinly $t(\alpha)^+$-approximated by $\Coll(t(\alpha)^+,\la)^V$ over $V[L_\rho][G_t]$. 
		We wish to apply Lemma \ref{lem:M-S} to show that $\Coll(t(\alpha)^+, \la)^V$ has the thin $t(\alpha)^+$-approximation property over $V[L_\rho][G_t]$. 
		
		To do this, we need to refactor $\PP_t$. Select $\beta$ such that $\w < \beta < t(\alpha)$. We factor $\PP_t$ as
		
		\[\PP_t = \left(\prod_{i < \beta+1} E(t(i)^{++}, t(i+1))\right) \times \left(\prod_{\beta+1\leq i < \theta}  E(t(i)^{++}, t(i+1))\right).\]
		We will denote the first component by $\PP_{t,\leq \beta}$ and the second by $\PP_{t, >\beta}$, with respective generics $G_{t,\leq \beta}$ and $G_{t, >\beta}$.
		
		Noting that $\PP_{t,> \beta}$ is $t(\beta+1)^{++}$-closed, in $V[G_{t,> \beta}]$ the collapse poset $\Coll(t(\alpha)^+, \la)^V$ remains $t(\beta+1)^{++}$-closed.
		$\PP_{t,\leq \beta}\times \LL_\rho$ has size $<t(\beta+1)^{++}$; this remains true in $V[G_{t,\geq \beta}$.
		
		Applying Lemma \ref{lem:M-S} in $V[G_{t,>\beta}]$, we conclude that in $V[G_{t,>\beta}][G_{t,\leq \beta}][L_\rho] = V[L_\rho][G_t]$, $\Coll(t(\alpha)^+,\la)^V$ has the thin $t(\alpha)^+$-approximation property. Then the collapse could not have added the branch $b$, so $b\in V[L_\rho][G_t]$.
	\end{proof}
Since ineffability is downwards absolute, $b$ remains an ineffable branch in $V[L_\rho][G_t]$. This completes the proof.

\section{Segments of multiple cofinalities}\label{s:manycofs}
The arguments in the previous section (and in Hayut and Golshani's original paper) are very similar to the techniques used by the author in \cite{adkisson:manycofs}, where the strong tree property was forced at successors of singulars of many cofinalities simultaneously. In fact, the arguments in that paper were directly inspired by Golshani and Hayut's proof. In both constructions, the (strong) tree property is obtained in a larger model that projects down to the desired forcing. Moreover, the constructions in each setting project to a \emph{different part} of the target forcing. To obtain the tree property for many singulars of the same cofinality, we look at forcings projecting onto the second component of the forcing, the product of collapses between the supercompacts. To obtain the tree property at successors of singulars of many cofinalities, we instead project onto the first part of the forcing, the piece that makes $\ka_0$ become accessible.

Since the projections are on completely separate parts of the forcing, they do not interact, so it is natural to combine these constructions. In this section we build a forcing that will obtain the strong tree property for segments of successors of singular cardinals of multiple fixed cofinalities simultaneously. As in \cite{adkisson:manycofs}, the final model will have a new $\w_n$ for each $1\leq n<\w$ that will be selected partway through the construction (and depend not only on the previous choices of $\w_i$, but also on the desired lengths of the previous segments). This makes the equivalent of ``arbitrarily long countable segment" for uncountable cofinalities slightly more awkward to state. The correct generalization is the following:

\begin{theorem}
Consider the following game of length $\w$. Player I plays cardinals $\theta_i$ and Player II plays cardinals $\mu_i$ such that the following properties are satisfied:
\begin{enumerate}
	\item $\theta_0 < \w_1$
	\item $\theta_i < \mu_{i-1}^{++}$ for $0<i<\w$
	\item $\langle \mu_i \mid i < \w\rangle$ is an increasing sequence.
\end{enumerate}

Player II wins the game if there is a generic extension with the following properties:
\begin{itemize}
	\item For all $i<\w$, $\mu_i^+ = \w_i$.
	\item The strong tree property holds at the first $\theta_0$-many successors of singulars of countable cofinality, except for $\aleph_{\w+1}$.
	\item For all $0<i<\w$, the strong tree property holds for the first $\theta_i$-many successors of singulars of cofinality $\w_i$.
\end{itemize}

Suppose there is a set $S$ of supercompact cardinals with $|S| \geq \min(S)^+$. Then Player II has a winning strategy in this game.
\end{theorem}
\begin{proof}

Let $\ka_0 = \min(S)$.
For each $\alpha \in \lim(S)$, fix an increasing continuous cofinal sequence $s_\alpha\subseteq S\cup \lim(S)$ in $\alpha$ of order type $\cf(\alpha)$, starting at $\ka_0$.

We will require the following generalization of Fact \ref{fact:t}.

\begin{lemma}\label{lem:t_manycofs_weak}
	Fix cardinals $\theta$ and $\mu$ such that $\theta < \mu^{++}$, and let $\gamma \in S$. Let $\alpha_0<\mu$, and suppose that for all clubs $E\subseteq \sup(S)$, $\sup\{s_\beta(\alpha_0) \mid \beta \in S\cap E\} = \sup(S).$
	Let $S'$ be a stationary subset of $\lim(S)\cap \sup(S) \cap \cof(\mu^+)$ and $n_0 < \w$ such that the first $n_0$ elements in $s_\alpha$ are the same for all $\alpha \in S'$.
	
	Then there is an $\mu^+$-closed sequence $t\subseteq S'$ with $\min(t) >\gamma$ and $\ot(t) = \theta$ such that for all $\alpha \in \lim(t)\cap \cof(\w_i)$, $\alpha \in S'$ and $s_\alpha \subseteq t$.
\end{lemma}
\begin{proof}
	We proceed by induction on $\theta$. We slightly strengthen the inductive hypothesis to require that for all $\gamma < \sup(S)$, there is $t$ meeting the conditions of the lemma, such that the first $n_0$ elements of $t$ are the fixed first $n_0$ elements of the sequences $s_\alpha$ for $\alpha$ in $S$, and $t(n_0+1) \geq \gamma$. 
	
	If $\theta = \alpha + \delta$ for $\delta < \mu^+$, then we apply the inductive hypothesis to $\alpha$ to obtain a sequence $t_\alpha$. Let $t = t_\alpha \cup [\alpha, \theta)$. Since we only require $\w_i$-closure, this sequence satisfies our requirements.
	
	Suppose $\cf(\theta) = \tau$ for some $\tau < \mu^+$, and $\theta$ is the limit of $\tau$-many ordinals, each of cofinality $\w_{i}$. Let $\langle \delta_\alpha \mid \alpha<\mu^+\rangle$ be such a sequence. By the inductive hypothesis, for each $\delta_\alpha$, we have a sequence $t_\alpha$ meeting the requirements with $\min(t_{\alpha+1}) > \sup t_\alpha$. Let $t = \bigcup_{\alpha<\tau} t_\alpha$.
	
	We move now to the case when $\theta$ has large cofinality.
	Suppose $\theta = \alpha + \mu^+$. We again apply the inductive hypothesis to obtain a sequence $t_\alpha$. Select $\beta \in S'$ with $\beta > \sup(t_\alpha)$. Let $t = t_\alpha \cup \{\sup(t_\alpha)\} \cup s_{\beta}$.
	
	Now suppose $\theta$ is a limit of limit ordinals, each with cofinality $\mu^+$. Let $\langle \delta_\alpha \mid \alpha<\mu^+\rangle$ be such a cofinal sequence. Fix some $\gamma_0 < \sup(S)$.
	
	Select $\nu \in S'$ such that for all $\zeta < \theta$ and all $\beta < \nu$ there is $t\subseteq \nu$ of order type $\zeta+1$ meeting the conditions of the lemma, such that $t(n_0+1) \geq \beta$ and $s_\nu(n_0+1) \geq \gamma_0$. We can do this by the stationarity of $S'$.
	
	Enumerate $s_\nu$ by $\langle \beta_\alpha \mid \alpha<\w_i\rangle$. Inducting on $\alpha$, we construct sequences $t_\alpha$ meeting the conditions of the lemma with the following additional properties:
	\begin{enumerate}
		\item $\ot (t_n) \geq \delta_n$
		\item $t_{\alpha}(n_0+1) > \sup_{\alpha' < \alpha} s_{\nu}(\alpha)$
		\item $t_{\alpha}(n_0+1) > \sup_{\alpha' < \alpha} \max(t_{\alpha'})$
		\item $t_0(n_0) \geq \gamma_0$
		\item $t_\alpha \subseteq \nu$.
	\end{enumerate}
	
	We define our final sequence to be $t = s_\nu \cup \bigcup_{\alpha < \mu^+} t_\alpha$. This will have order type at least $\theta$. Moreover, the $\mu^+$-limit points of this sequence consist of $\nu$ and the $\mu^+$-limit points of each $t_\alpha$; since these are closed by assumption, we meet the conditions of the lemma.
\end{proof}

We will define Player II's winning strategy (that is, the selection of the $\mu_i$'s) inductively. Along the way, we will define a number of auxiliary forcings that will be used to construct Player II's final model and prove that the appropriate cardinals have the strong tree property in this model.
	
Our argument is very similar to the previous section, combined with the arguments in \cite{adkisson:manycofs}. The extra layer of complexity introduced by multiple cofinalities is that when we select each $\mu_i$, not only do we need to select a cardinal that works for stationary many limits of supercompact cardinals, but we also need to select $\mu_i$ that will accommodate all later selections $\mu_j, i<j$. We do this by using a larger forcing that depends only on $\mu_i$ and the choices we've already made, and will project onto to the final model regardless of the later choices of $\mu_j$.
	 
	First, let's focus on countable cofinality.
	As before, let $\ka_0 = \min(S)$, and $D = \lim(S)\cap \sup(S) \cap (\cof(\w))$. For each $\alpha \in D$, let $s_\alpha$ be the pre-selected $\w$-sequence cofinal in $\alpha$.
	
	For each $\mu_0 < \ka_0$, define
	\[\LL^0(\mu_0) \defeq \Coll(\w, \mu_0) \times \Coll(\mu_0^+, <\ka_0) \times \prod_{\mu_0 < \alpha < \ka_0} \Coll(\mu_0^{++}, \ka_0).\]
	We select this poset because it will project to a more complicated poset that incorporates future selections, which will appear when we build the final model.
	
	For each $\alpha \in D$ and $\la \geq \ka_0$, we define the forcing $\PP^0_{\alpha,\la}$ as in the previous section:
	
	\[\PP^0_{\alpha,\la} = \left(\prod_{n<\w} E(s_\alpha(n)^{++}, s_\alpha(n+1))\right) \times E(\alpha^{++}, \sup(S))\times \Coll(\alpha^+, \la).\]
	
	For each such pair $(\alpha, \la)$, we apply Lemma \ref{lem:gen2cardstrTP}, noting that $|\Coll(\w, \mu_0)|\leq \mu_0$ and the remainder of $\LL^0(\mu_0)$ is $\mu_0^+$-closed. The required lifting behavior at each $\ka_\alpha$ is given by Lemma \ref{lem:liftingembeddings}. By Lemma \ref{lem:gen2cardstrTP}, we see that there is a cardinal $\mu_0^{\alpha, \la}$ such that the poset $\LL^0(\mu_0^{\alpha, \la}) \times \PP^0_{\alpha, \la}$ forces $\TP(\alpha^+, \alpha^+)$.
	
	For each fixed $\alpha \in D_0$, there is some $\mu_0^\alpha < \ka_0$ such that for unboundedly many $\la > \sup(S)$, $\mu_0^{\alpha, \la} = \rho_\alpha$. Applying Fodor's Lemma, we obtain $\mu < \ka$ and a stationary set $S' \subseteq D_0$, such that for all $\alpha \in S'$, $\mu_0^\alpha = \mu_0$.
	
	Fix $\theta_0 < \w_1$. Applying Lemma \ref{lem:t_manycofs_weak}, we obtain a closed subsequence $t_0\subseteq S\cup S'$ of order type $\theta_0$, such that for all $\alpha \in \lim(t_0)$, $\alpha \in S'$ and $s_\alpha \subseteq t_0$.
	
	Define $\PP_{t_0}^0:$
	
	\[\PP_{t_0}^0 \defeq \prod_{i<\theta_0} E(t_0(i)^{++}, t_0(i+1)).\]
	
	Now, working inductively, suppose the players have defined sequences $\langle \mu_n \mid n< i \rangle$ and $\langle \theta_n \mid n<i\rangle$ as in the description of the game. Suppose Player II has also defined sequences $\langle t_n \mid n<i\rangle$.
	
	Player I makes their next move, giving Player II a fixed $\theta_i < \mu_{i-1}^{++}$.
	
	For all $\mu_i < \ka_0$, we define a forcing $\LL^i(\mu_i)$ as follows. Define
	\[\LL^i_0(\mu_i) \defeq \Coll(\w, \mu_0) \times \left(\prod_{n<i-1} \Coll(\mu_n^+, \mu_{n+1})\right)\times \Coll(\mu_{i-1}^+, \mu_i)\]
	and
	\[\LL^i_1(\mu_i) \defeq \Coll(\mu_i^+, <\ka_0) \times	\prod_{\mu_0 < \alpha < \ka_0} \Coll(\mu_0^{++}, \ka_0).\]
	Finally, we define
	\[\LL^i(\mu_i) = \LL^i_0(\mu_i) \times \LL^i_1(\mu_i).\]
	
	Let $D_i = (\lim(S) \cap \sup(S)\cap \cof(\mu_{i-1}^+))\sm \sup(S_{i-1})$. For each $\alpha \in D$, let $s_\alpha$ be the preselected $(\mu_{i-1}^+)$-sequence $s_\alpha$ cofinal in $\alpha$.
	
	For all $\alpha \in D_i$ and all $\la \geq \ka_0$, define the forcing $\PP_{\alpha, \la}^i$ as follows:
	
	\[\PP_{\alpha, \la}^i = E(\ka_0^{++}, \sup(t_{i-1})) \times \left(\prod_{\beta < \mu_{i-1}^+} E(s_\alpha(\beta)^{++}, s_\alpha(\beta+1))\right) \times E(\alpha^{++}, \sup(S)) \times \Coll(\alpha^+, \la). \]
	
	For all such pairs $(\alpha, \la)$, we wish to apply Lemma \ref{lem:gen2cardstrTP}. We can do so because $\LL^i_0(\mu_i)$ has size $\leq \mu_i$, and $\LL^i_1$ is $\mu_i^+$-closed. As before, the lifting of embeddings is given by Lemma \ref{lem:liftingembeddings}. Applying Lemma \ref{lem:gen2cardstrTP}, for each such pair $(\alpha, \la)$ we obtain a cardinal $\mu_i^{\alpha, \la}$ such that the poset $\LL^i(\mu_i^{\alpha,\la})\times \PP_{\alpha, \la}^i$ forces the strong tree property at $\alpha^+$. For each $\alpha$, we obtain $\mu_i^\alpha$ such that $\mu_i^{\alpha, \la} = \mu_i^\alpha$ for unboundedly many $\la$. By Fodor's Lemma we can thin $D_i$ to a stationary subset $S'_i$ such that for all $\alpha \in S'$, $\mu_\alpha = \mu$ for some fixed $\mu < \ka_0$. 
	Player II makes their move, selecting this fixed $\mu$ to be $\mu_i$.
	
	Applying Lemma \ref{lem:t_manycofs_weak}, we obtain a $\mu_{i-1}^+$-closed sequence $t_i\subseteq S\cup S'$ of order type $\theta_i$ such that for all $\alpha \in \lim(t_i)\cap \cof(\mu_{i-1}^+)$, $\alpha \in S'$ and $s_\alpha \subseteq t_i$.
	
	Finally, we define $\PP_{t_i}^i$:
	
	\[\PP_{t_i}^i\defeq \prod_{\beta<\theta_i} E(t_i(\beta)^{++}, t_i(\beta+1)).\]
	
	Now we prove that this strategy gives a win for Player II. That is, we will define a forcing with the desired properties with respect to the sequences $\langle \mu_i \mid i<\w\rangle$ and $\langle \theta_i \mid i<\w\rangle$ produced by this play of the game. Define
	
	\[\LL \defeq \Coll(\w, \mu_0) \times \left(\prod_{n<\w} \Coll(\mu_n^+, \mu_{n+1})\right) \times \Coll((\sup_{n<\w} \mu_n)^+, <\ka_0)\]
	and
	\[\PP = \prod_{i<\w} \PP^i_{t_i}.\]
	Let $L$ be generic for $\LL$, and $G$ be generic for $\PP$.
	We note the cardinal structure in this model. In $V[L]$, $(\mu_i^+)^V = \aleph_i$, with $\sup_{i<\w} \mu_i = \aleph_\w$, and $\ka_0 = \aleph_{\w+2}$. 
	
	To finish the proof of the theorem, it suffices to, for all $i<\w$ and all $\alpha < \theta_i$, verify the tree property at the $\alpha$-th successor of a singular of cofinality $\w_i$. This cardinal will be $t_i(\alpha)$.
	
	\begin{lemma}\label{lem:projection_i}
		Fix $i < \w$, $\alpha < \mu_i$, and $\la \geq t_i(\alpha)^+$. Then $\LL^i(\mu_i) \times \PP_{t(\alpha)^+,\la}^i$ projects onto $\LL\times \PP\times \Coll(t(\alpha)^+,\la)$.
	\end{lemma}
	\begin{proof}
		We verify the projection in pieces. First, we claim that $\LL_i(\mu_i)$ projects onto $\LL$. The arguments here are very similar to those in \cite{adkisson:manycofs}. Recall that we split $\LL_i(\mu_i)$ into two pieces,
			\[\LL^i_0(\mu_i) = \Coll(\w, \mu_0) \times \left(\prod_{n<i-1} \Coll(\mu_n^+, \mu_{n+1})\right)\times \Coll(\mu_{i-1}^+, \mu_i)\]
		and
		\[\LL^i_1(\mu_i) = \Coll(\mu_i^+, <\ka_0) \times \prod_{\mu_0 < \alpha < \ka_0} \Coll(\mu_0^{++}, \ka_0).\]
		
		Similarly, we can split $\LL$ into two pieces:
		
		\[\LL_0 \defeq \Coll(\w, \mu_0) \times \prod_{n<i} \Coll(\mu_n^+, \mu_{n+1})\]
		and
		\[\LL_1 = \left(\prod_{i\leq n < \w} \Coll(\mu_n^+, \mu_{n+1})\right) \times \Coll((\sup_{n<\w} \mu_n)^+, <\ka_0).\]
		
		$\LL_0^i(\mu_i)$ projects onto $\LL_0$ via the identity. $\LL_1^i(\mu_i)$ projects onto $\LL_1$: the first part projects by Lemma \ref{lem:colabsorption}, and the second part projects by simply taking the appropriate coordinate of the product. We conclude that $\LL_i(\mu_i)$ projects onto $\LL$. The quotient of this projection has size $<\ka_0^{++}$.
		 
		Next, we claim that $E(\ka_0^{++}, \sup(t_{i-1}))$ projects onto $\PP_{<i}$. This follows from Lemma \ref{lem:projections}. Both of these posets %have size $<t_i(\alpha)^+$,
		are $\ka_0^{++}$-closed, so the corresponding quotients must be as well.
		By the same arguments as in the previous section, $\PP^i_{t_i(\alpha)^+,\la}$ projects onto  $\PP_{t_i}^i$, with a $\ka_0^{++}$-closed quotient.
		Finally, $E(t_i(\alpha)^{++}, \sup(S))$ projects onto $\PP_{>i}$, with a $t_i(\alpha)^{++}$ closed quotient.
	\end{proof}
	
	\begin{lemma}
		Fix $i<\w$ and $\alpha < \mu_i$. Then in $V[L][G]$, the strong tree property holds at $t_i(\alpha)^+$.
	\end{lemma}
	\begin{proof}
		Fix $i < \w$ and $\alpha < \mu_i$. Fix $\la \geq t_i(\alpha)^+$ such that $\la$ is strong regular and $\mu_i^{\alpha,\la} - \mu_i$. Since this holds for unboundedly many $\la$, by Fact \ref{fact:ITPgoesdownwards} it suffices to verify $\TP(t_i(\alpha)^+, \la)$ for only these $\la$.
		
		Let $d$ be a thin $\mc{P}_{t_i(\alpha)^+}(\la)$-list. Let $K$ be generic for $\Coll(t_i(\alpha)^+,\la)^V$. Since $\Coll(t_i(\alpha)^+,\la)^V$ is $<t_i(\alpha)^+$-distributive by Lemma \ref{lem:easton}, $\mc{P}_{t_i(\alpha)^+}(\la)^{V[L][G]} = \mc{P}_{t_i(\alpha)^+}(\la)^{V[L][G][K]}$. (Since the two sets are the same, we will omit the superscripts.)
		
		In $V[L][G][K]$, $\la$ is an ordinal of cardinality and cofinality $t_i(\alpha)^+$. In particular, $\mc{P}_{t_i(\alpha)^+}(\la)$ is isomorphic to $\mc{P}_{t_i(\alpha)^+}(t_i(\alpha)^+)$ in this model. It follows that there is a thin $\mc{P}_{t_i(\alpha)^+}(t_i(\alpha)^+)$-list $d'$ isomorphic to $d$.
		
		Let $d^* = \{d'_\alpha \mid \alpha < t_i(\alpha)^+\}$; that is, $d'$ restricted to ordinals. Note that $d^*$ is a $t_i(\alpha)^+$-list; since $t_i(\alpha)^+$ is club in $\mc{P}_{t_i(\alpha)^+}(t_i(\alpha)^+)$, $d^*$ has a cofinal branch if and only if $d'$ (and thus $d$) has a cofinal branch.
		
		By Lemma \ref{lem:projection_i}, $\LL^i_{\mu_i} \times \PP^i_{t_i(\alpha), \la}$ projects onto $\LL\times \PP$. Let $L' \times G'$ be a generic for $\LL^i_{\mu_i} \times \PP^i_{t_i(\alpha),\la}$ projecting onto $L\times G$. By construction (and our specific choice of $\mu_i$), $\TP(t_i(\alpha)^+,t_i(\alpha)^+)$ holds in $V[L^*][G^*][K]$.
		
		Since being a $t_i(\alpha)^+$-list is upwards absolute, $d^*$ remains a $t_i(\alpha)^+$-list in this model. It follows that $d^*$ has a cofinal branch $b^*$ in $V[L^*][G^*][K]$.
		
		\begin{claim}
		The branch $b^*$ is in $V[L][G][K]$.
		\end{claim}
		\begin{proof}
			By the projection analysis in Lemma \ref{lem:projection_i}, $V[L'][G'][K] = V[L][G][K][L^*][G^*_0][G^*_1]$, where $L^*\times G_0^*\times G_1^*$ is generic for the quotient $(\LL^i_{\mu_i}\times \PP^i_{t_i(\alpha)})/(\LL\times \PP)$. $L'\times G'_0$ is generic for a forcing of size $<\ka_{0}^{++}$, and $G_1'$ is generic for a $\ka_0^{++}$-closed forcing. By Lemma \ref{lem:M-S}, the closed part has the thin $t_i(\alpha)^+$-approximation property over the rest.
			
			By Lemma \ref{lem:cctoapprox}, $b^*$ is thinly approximated by this part of the quotient over $V[L][G][K][L^*][G^*_0]$. We conclude that the closed part can't have added $b$, so $b^*$ is in $V[L][G][K][L^*][G^*_0]$.
			
			The remainder of the forcing has size at most $\ka_0^+$. Thus by Lemma \ref{lem:squarecctoapproxprop} it has the approximation property over $V[L][G][K]$. By Fact \ref{fact:branchapprox}, $b^*$ is thinly approximated (by any forcing) over this model. We conclude that $b^* \in V[L][G][K]$.
		\end{proof}
		
		Since $d^*$ has a cofinal branch if and only if $d$ has a cofinal branch, we conclude that $d$ has a cofinal branch $b$ in $V[L][G][K]$. Since $d$ is defined in $V[L][G]$, it must be thinly $t_i(\alpha)^+$-approximated by any forcing. 
		
		As before, $b$ is thinly $t(\alpha)^+$-approximated by $\Coll(t_i(\alpha)^+, \la)^V$ over $V[L][G]$.
		It remains to show that $\Coll(t_i(\alpha)^+, \la)^V$ has the thin $t_i(\alpha)^+$-approximation property and thus could not have added a cofinal branch.
		
		Fix $\delta$ such that $\mu_{i-1}^{+} < \delta < \theta_i$. Factor $\PP^i_{t_i}$ as
		\[\PP_{t_i}^i = \left(\prod_{\beta \leq \delta} E(t_i(\beta)^{++}, t_i(\beta+1))\right)\times \left( \prod_{\delta < \beta < \theta_i} E(t_i(\beta)^{++}, t_i(\beta+1))\right).\]
		Let $\PP^i_{t_i, \leq\delta}$ denote the first component and $\PP^i_{t_i, >\delta}$ the second component, with respective generics $G^i_{t_i,\leq \delta}$ and $G^i_{t_i, >\delta}$. Since $\PP^i_{t_i, \leq\delta}$ is $t(\delta+1)^{++}$-closed, in $V[G^i_{t_i, >\delta}]$ the collapse $\Coll(t_i(\alpha)^+, \la)^V$ is still $t(\delta+1)^{++}$-closed. $\PP_{t_i, \leq\delta}\times \LL$ has size $<t_i(\delta+1)^{++}$.  Applying Lemma \ref{lem:M-S}, we see that $\Coll(t(\alpha)^+, \la)^V$ has the thin $\nu^+$ approximation property in the model $V[G^i_{t_i, >\delta}][G^i_{t_i, \leq \delta}][L] = V[L][G]$. Then $b$ could not have been added by the collapse.
		
		We conclude that $d$ must have a cofinal branch in $V[L][G]$ as desired.
	\end{proof}
	This concludes the proof; since Player II can always select $\mu_i$ such that the strong tree property holds at each $t_i(\alpha)^+$ for $\alpha < \theta_i$ in some generic extension, they have a winning strategy in the game.
\end{proof}

\section{Acknowledgements}
The author is grateful to Itay Neeman for many productive and helpful conversations.

\bibliography{bib}
\bibliographystyle{plain}
\end{document}